\documentclass{cedram-aif}
\usepackage{amssymb, latexsym, mathrsfs, amsmath, amsfonts, hyperref, enumitem, fancyhdr}

\hypersetup{
  colorlinks,
  citecolor=black,
  linkcolor=black,
  urlcolor=blue}
\theoremstyle{plain}

  {\end{enumerate}}

\newcommand{\D}{\mathbb{D}}

\newcommand{\R}{\mathbb{R}}
\newcommand{\h}{\mathbb{H}}

\newcommand{\N}{\mathbb{N}}
\newcommand{\T}{\mathbb{T}}

\newcommand{\C}{\mathcal{C}}
\newcommand{\VK}{Varopoulos--Kaijser }

\hypersetup{
  colorlinks,
  citecolor=blue,
  linkcolor=blue,
  urlcolor=blue}

\begin{document}
\title[On a Question of N. Th. Varopoulos
and the constant $C_2(n)$]{On a Question of N. Th. Varopoulos \\
and the constant $C_2(n)$}


\thanks{The first named author is supported by National Postdoctoral Fellowship (Ref. No. PDF/2016/000094), SERB, DST, Government of India.}

\thanks{The second named author is supported by Council for Scientific and Industrial Research, MHRD, Government of India.}

\author{\firstname{Rajeev} \lastname{Gupta}}
\address{Department of Mathematics and Statistics,\\ Indian Institute of Technology, Kanpur-208016}
\email{rajeevg@iitk.ac.in}

\author{\firstname{Samya} \middlename{K.} \lastname{Ray}}
\address{Department of Mathematics and Statistics,\\ Indian Institute of Technology, Kanpur-208016}
\email{samya@iitk.ac.in}

\pagestyle{headings}

\begin{abstract}
Let $\mathbb C_k[Z_1,\ldots, Z_n]$ denote 
the set of all polynomials of degree 
at most $k$ in $n$ complex variables 
and $\mathscr{C}_n$ denote the set of all 
$n$ - tuple $\boldsymbol T=(T_1,\ldots,T_n)$ of commuting contractions 
on some Hilbert space $\mathbb{H}.$ The interesting inequality
$$K_{G}^{\mathbb C}\leq \lim_{n\to \infty}C_2(n)\leq  2 K^\mathbb C_G,$$ 
where 
\[C_k(n)=\sup\big\{\|p(\boldsymbol T)\|:\|p\|_{\D^n,\infty}\leq 1, p\in \mathbb C_k[Z_1,\ldots,Z_n],\boldsymbol T\in\mathscr{C}_n \big\}\]
and $K_{G}^{\mathbb C}$ is the complex Grothendieck constant, 
is due to Varopoulos. 
We answer a long--standing question by showing
that the limit $\lim_{n\to\infty} \frac{C_2(n)}{K^\mathbb C_G}$ is strictly bigger than $1.$
Let $\mathbb C_2^s[Z_1,\ldots , Z_n]$ denote the set of all complex valued homogeneous polynomials $p(z_1,\ldots,z_n)$ $=\sum_{j,k=1}^{n}a_{jk}z_jz_k$
of degree two in $n$ - variables, where $(\!(a_{jk})\!)$ is a $n\times n$ complex symmetric matrix. 
For each $n\in\mathbb{N},$ define the linear map $\mathscr{A}_n:\big (\mathbb C_2^s[Z_1,\ldots , Z_n],\|\cdot\|_{\mathbb D^n, \infty}\big ) \to \big (M_n, \|\cdot \|_{\infty \to 1}\big )$ to be $\mathscr{A}_n\big (p) = (\!(a_{jk})\!).$ 
We show that the supremum (over $n$) of the norm of the operators 
$\mathscr{A}_n;\,n\in\mathbb{N},$ is bounded below by  the constant $\pi^2/8.$ 
Using a class of operators, first introduced by Varopoulos, 
we also construct a large class of explicit polynomials 
for which the von Neumann inequality fails. 
We prove that the original Varopoulos--Kaijser polynomial 
is extremal among a, suitably chosen, large class of homogeneous 
polynomials of degree two. We also study the behaviour of the 
constant $C_k(n)$ as $n \to \infty.$
\end{abstract}
\begin{altabstract}
Soit $\mathbb C_k[Z_1,\ldots, Z_n]$
l'ensemble de tous les polyn\^{o}mes de degr\'e 
au plus $k$ dans $n$ variables complexes
et $\mathscr {C}_n$ indiquent l'ensemble de tous
$n$ - tuple $\boldsymbol T=(T_1,\ldots,T_n)$ de contractions de navettage
sur un espace de Hilbert $\mathbb{H}.$ L'in\'{e}galit\'{e} int\'{e}ressante
$$K_{G}^{\mathbb C} \leq \lim_ {n \to \infty} C_2(n)\leq 2 K^\mathbb C_G, $$
o\`{u}
\[C_k(n)=\sup\big\{\|p(\boldsymbol T)\|:\|p\|_{\D^n,\infty}\leq 1, p\in \mathbb C_k[Z_1,\ldots,Z_n],\boldsymbol T\in\mathscr{C}_n \big\}\]
et $ K_{G}^{\mathbb C}$ est la constante complexe de Grothendieck,
est due \`{a} Varopoulos.
Nous r\'{e}pondons \`{a} une question de longue date en montrant
que $\lim_{n\to\infty} \frac{C_2(n)}{K^\mathbb C_G}$ est strictement plus grand que $1$.
Soit $\mathbb C_2^s [Z_1, \ldots, Z_n]$ l'ensemble des polyn\^{o}mes homog\`{e}nes complexes $ p (z_1, \ldots, z_n) =~~~~$ $\sum_{j, k = 1}^{n} a_{jk} z_jz_k $
de degr\'{e} deux dans $n$ - variables, o\'{u} $(\!(a_{jk})\!)$ est une matrice sym\'{e}trique complexe $n\times n$.
Pour chaque $ n \in \mathbb {N}, $ d\'{e}finit la carte lin\'{e}aire 
$\mathscr{A}_n:\big (\mathbb C_2^s[Z_1,\ldots , Z_n],
\|\cdot\|_{\mathbb D^n, \infty}\big ) \to \big (M_n, 
\|\cdot \|_{\infty \to 1}\big )$ \'{a} 
$\mathscr{A}_n\big (p) = (\!(a_{jk})\!).$ 
Nous montrons que le supremum (plus de $n$) de la norme des op\'{e}rateurs $\mathscr {A}_n; \, n \ \in \mathbb{N},$ est limit\'{e} ci-dessous par la constante $\pi^2/8.$
En utilisant une classe d'op\'{e}rateurs, introduit par Varopoulos, nous construisons aussi une grande classe de polyn\^{o}mes explicites pour lequel l'in\'{e}galit\'{e} de von Neumann \'{e}choue.
Nous prouvons que le polyn\^{o}me original Varopoulos-Kaijser est extr\'{e}male parmi une grande classe d'homog\`{e}ne convenablement choisie polyn\^{o}mes de degr\'{e} deux. Nous \'{e}tudions \'{e}galement le comportement de constante $C_k(n)$ as $n \to \infty.$
\end{altabstract}

\maketitle

\section{Introduction} 
Let $\mathbb C[Z_1,\ldots,Z_n]$ denote 
the set of all complex valued polynomials in $n$ - variables.
For any continuous function $f:X\to \mathbb C,$ 
on a compact set $X,$ we let 
$\|f\|_{X,\infty}$ denote its supremum norm, namely,   
\[\|f\|_{X,\infty}=\sup\{|f(x)|:x\in X\}.\]  
Let $\mathbb{D}^n$ denote the unit polydisc in $\mathbb{C}^n.$ 
The von Neumann inequality \cite{vN} states that 
$\|p(T)\|\leq \|p\|_{\D,\infty}$ 
for all $p\in\mathbb{C}[Z]$ and for any contraction $T$ on a complex Hilbert space. 
For any pair of commuting contractions $T_1,T_2$, a generalization of the von Neumann inequality:
$\|p(T_1,T_2)\|\leq \|p\|_{\D^2,\infty},$ $p\in\mathbb{C}[Z_1,Z_2]$ 
follows from a deep theorem of Ando \cite{ando} on unitary dilation of a pair of commuting contractions. For $n\in\N,$ let 
$\mathscr{C}_n$ denote the set of all $n$ - tuple $\boldsymbol T=(T_1,\ldots,T_n)$ of commuting contractions 
on some Hilbert space $\mathbb{H}.$ In the paper \cite{V1}, Varopoulos showed 
that the von Neumann inequality 
fails for $\boldsymbol T$ in $\mathscr{C}_n,$ $n > 2.$ 
He along with Kaijser \cite{V1} and 
simultaneously Crabb and Davie \cite{CD} 
produced an explicit example of 
three commuting contractions $T_1,T_2,T_3$ 
and a polynomial $p$ for which 
$\|p(T_1,T_2,T_3)\| > \|p\|_{\D^3,\infty}.$ 
For a fixed $k\in\mathbb{N},$ define  (see \cite{V2} and \cite[page 24]{Pisier}):
\[C_k(n)=\sup\big\{\|p(\boldsymbol T)\|:\|p\|_{\D^n,\infty}\leq 1, p\in \mathbb C_k[Z_1,\ldots,Z_n],\,\boldsymbol T \in\mathscr C_n \big\}\]
and let $C(n)$ denote $\lim_{k\to \infty} C_k(n).$

Since the counterexample to the von Neumann inequality in three variables, 
due to Varopoulos and Kaijser \cite{V1}, 
involves a (explicit) homogeneous polynomial of degree two therefore 
$C_2(3) >1.$ From the von Neumann inequality and its generalization to two variables, it follows that  $C(1)=C(2)=1.$   
In the paper \cite{V2}, Varopoulos shows that 
\begin{equation}\label{bound for C_2}
K_G^\mathbb C \leq \lim_{n\to \infty} C_2(n) \leq 2 K_G^\mathbb C,
\end{equation}
where $K_G^\mathbb C$ is the complex Grothendieck constant defined below. 
\begin{defi}[Grothendieck Constant]
	For a complex $($real$)$ array $A:=\big (\!\!\big ( a_{j k} \big ) \!\!\big )_{n\times n},$ define the following norm 
	\begin{equation}\label{hypothesis}
		\|A\|_{\infty \to 1} := \sup\big\{|\langle A v , w\rangle| :\|v\|_{\ell^\infty(n)}\leq 1, \|w\|_{\ell^\infty(n)} \leq 1\big\},
	\end{equation}
	where $v$ and $w$ are vectors in $\mathbb C^n$ $($resp. $\mathbb R^n)$. 
	There exists a finite constant $K>0$  
	such that for any choice of unit vectors 
	$(x_j)_1^n$ and $(y_k)_1^n$ in a complex $($resp. real$)$ Hilbert space $\mathbb{H}$, 
	we have
	\begin{equation}\label{conclusion}
		\Big |\sum_{j,k=1}^na_{jk}\langle x_j,y_k \rangle\Big |   
		\leq K \|A\|_{\infty \to 1}  
	\end{equation}
	for all $n\in\mathbb{N}$ and $A= \big (\!\!\big ( a_{j k} \big ) \!\!\big ).$
	The least such constant is denoted by $K_G$ 
	and is known as the Grothendieck constant. 
	Note that the definition of $K_G$ depends  on the underlying field. 
	When it is the field of complex numbers $\mathbb C$ $(\mbox{\rm resp. }\mathbb{R})$,   
	this constant is known as the 
	complex (resp. real) Grothendieck constant and is denoted by $K_G^\mathbb C$ $($resp. $(K_G^\mathbb{R}))$. 
It is known that $1.338 < K_G^\mathbb C\leq 1.4049,$ and $1.66\leq K_G^\mathbb{R}\leq \frac{\pi}{2log(1+\sqrt{2})}$, see \cite[\S {4}]{GrPis}. 
Recently, it has been proved in \cite{Naor} that this upper bound of $K_G^\mathbb{R}$ is strict which settles a long--standing conjecture of Krivine \cite{Kr1,Kr2}.  
We refer the reader to \cite{BBGM} for some explicit computations of  this constant for small values of $n.$ 
For more on Grothedieck constant, we refer the reader to \cite{GrPis}.	
\end{defi}

Since it is known that $K_{G}^{\mathbb C}> 1,$ the inequality \eqref{bound for C_2} is 
yet another way to see that the von Neumann inequality fails eventually. 
We refer to the inequality \eqref{bound for C_2} as {\tt  the Varopoulos inequality}. 
In the paper \cite{V2}, Varopoulos had implicitly asked  
if 
$\lim_{n\to\infty}C_2(n)=K_G^\mathbb{C}?$  
Recently, first named author of this paper,  
has improved the Varopoulos inequality \eqref{bound for C_2}:
\begin{equation}\label{Improved bound for C_2}
K_G^\mathbb C \leq \lim_{n\to \infty} C_2(n) \leq \frac{3\sqrt{3}}{4} K_G^\mathbb C.
\end{equation}
This inequality is proved by first obtaining a bound for the second derivative of any holomorphic map 
$f:\D^n\to \D,$ namely, 
$\|D^2f(0)\|_{\infty\to 1} \leq \tfrac{3\sqrt{3}}{2}.$ 
In this paper, we answer the question of Varopoulos in the negative 
by improving the lower bound in the inequality \eqref{bound for C_2}. Indeed,  
we prove that 
$\lim_{n\to \infty}C_2(n)\geq 1.118 K_{G}^{\mathbb C}.$ 

In what follows, for each $p\in [1,\infty]$ and $n\in\mathbb{N},$ 
we denote the normed linear space $(\mathbb{C}^n,\|\cdot\|_p)$ 
by $\ell^p(n)$ and when the space is $(\mathbb{R}^n,\|\cdot\|_p)$ 
then we choose to denote it by $\ell^p_\mathbb{R}(n).$   
Let $\mathbb C_2^s[Z_1,\ldots , Z_n]$ denote the set of all homogeneous polynomials of degree two in $n$ - variables. 
A  homogeneous polynomial of degree two in $n$ - variables is of the form 
\[p(z_1,\ldots,z_n)=\sum_{j,k=1}^{n}a_{jk}z_jz_k,\]
where $A_p:=\big (\!\!\big ( a_{j k} \big ) \!\!\big )$ is a symmetric matrix associated to $p$. 
Define the map 
$\mathscr{A}_n:\mathbb C_2^s[Z_1,\ldots , Z_n] \to M^s_n$ 
by the rule $\mathscr{A}_n(p)=A_{p},$ 
where $M_n^s$ is the set of all symmetric matrices of order $n.$ 
Equip $\mathbb C_2^s[Z_1,\ldots , Z_n]$ with the supremum norm $\|\cdot\|_{\mathbb D^n,\infty}$ and $M_n^s$ with the norm  $\|\cdot\|_{\infty \to 1}.$ 
For each $n\in \N,$ $\|\mathscr{A}_n^{-1}\|\leq 1,$ 
therefore $\|\mathscr{A}_n\|\geq 1.$  

In \cite{RG}, it is shown that 
$\lim_{n\to \infty}\|\mathscr{A}_n\|\leq 3\sqrt{3}/4.$ 
In this paper, we prove $\lim_{n\to \infty}\|\mathscr{A}_n\|\geq \pi^2/8,$ 
improving the bound  
$\lim_{n\to \infty}\|\mathscr{A}_n\|\geq 1.2323,$ 
obtained earlier 
in an unpublished article by Holbrook and Schoch (see their related work \cite{HS10}). 

In Section \ref{MaximizingLemma}, we investigate, in some detail, the constant $C_2(n)$. 
We exhibit a large class of examples of \VK type and 
show that the original \VK example is extremal, in an appropriate sense,  in this large class of polynomials.

Let $M_{n}^{+}(\mathbb{C})$ $(\mbox{\rm resp. } M_{n}^{+}(\mathbb{R}))$ 
denote the set of all $n\times n$ complex (real) non-negative definite matrices.

\begin{defi}[Positive Grothendieck Constant]
Suppose $A:=\big ( \!\! \big (a_{jk}\big )\!\!\big )_{n\times n}$ 
is a complex $($real$)$ non-negative definite array 
then there exists $K>0,$ 
independent of $n$ and $A,$ such that \eqref{conclusion} holds. 
The least such constant is denoted by 
$K_G^+(\mathbb C)$ $(\mbox{\rm resp. }K_G^+(\R))$ and called 
complex $($resp. real$)$ Positive Grothendieck constant. 
The values of $K_G^+(\mathbb C)$ and $K_G^+(\R)$ 
are exactly $4/\pi$ and $\pi/2$ respectively, see \cite[page 259-260]{GrPis} and \cite[Remark following Theorem 5.4]{FactorizationPisier}. 
To the reader, we also refer \cite[Theorem 1.3]{PL} for these constants.
\end{defi}
The non-negative definite Grothendieck constant plays an important role in operator theory. We refer \cite[Theorem 1.9]{BM} for some important connections.

For any $n\times n$ complex matrix $A:=\big (\!\! \big (a_{jk}\big )\!\!\big ),$ 
we associate a homogeneous polynomial of degree two denoted by $p_{\!_A}$ and defined by 
$p_{\!_A}(z_1,\ldots,z_n)=\sum_{j,k=1}^n a_{jk} z_j z_k.$
Suppose $p$ is the following polynomial of degree at most two in $n$ - variables and of the form  
\[p(z_1,\ldots,z_n)=a_0+\sum_{j=1}^{n}a_jz_j + \sum_{j,k=1}^{n}a_{jk}z_jz_k\]
with $a_{jk}=a_{kj}$ 
(can be assumed without loss of generality) for all $j,k=1,\ldots,n.$ 
Corresponding to $p,$ one can define the following $(n+1)\times (n+1)$ 
symmetric complex matrix 
\begin{eqnarray}\label{Homogenization}
A(p)=
\left(
\begin{array}{ccccc}
	a_0 & \frac{1}{2} a_1 & \frac{1}{2}a_2 & \cdots & \frac{1}{2}a_n\\
	\frac{1}{2}a_1 & a_{11} & a_{12} & \cdots & a_{1n}\\
	\frac{1}{2}a_2 & a_{12} & a_{22} & \cdots & a_{2n}\\
	\vdots & \vdots & \vdots & & \vdots\\
	\frac{1}{2}a_n & a_{1n} & a_{2n} & \cdots & a_{nn}\\
\end{array}
\right).
\end{eqnarray}
It can be seen that $\|p\|_{\D^n,\infty}=\|p_{\!_{A(p)}}\|_{\D^{n+1},\infty}.$ 
We define the following quantity: 
\[C_{2}^{+}(n)=\sup\Big\{\frac{\!\!\!\|p_{\!_A}(\boldsymbol T)\|}{\ \ \ \ \, \|p_{\!_A}\|_{\D^n,\infty}}:0\neq A\in M_n^+(\mathbb{C}),\,  \boldsymbol T\in \ \mathscr C_n \Big\}\]
and let $C_{2}^{+}$ denote the quantity $\lim_{n\to \infty} C_{2}^{+}(n).$ 
It is clear from the definitions that $C_2(n)\geq C_{2}^{+}(n)$ for each $n.$ 
In this paper, we prove that $\lim_{n\to \infty}C_{2}^{+}(n)=\pi/2.$ 
Since $K_{G}^{\mathbb{C}}\leq 1.4049$ therefore 
$\lim_{n\to \infty}C_2(n)\geq \pi/2\geq 1.118K_{G}^{\mathbb C}.$

\section{Improvement of the Lower Bound in  the Varopoulos inequality}\label{Improvement}
In this section we improve the lower bound of  the Varopoulos inequality 
and as a consequence, we answer negatively a question of Varopoulos posed in the paper \cite{V2}. 
The following theorem (see \cite{RG}) concerns an improvement in 
the upper bound of 
the Varopoulos inequality.  

\begin{thm}
	Suppose $p$ is a polynomial of degree at most 2 in $n$ - variables 
	and $\boldsymbol T\in\mathscr{C}_n.$ Then
	$\|p(\boldsymbol T)\|
	\leq \frac{3\sqrt{3}}{4} K_{G}^{\mathbb C}\|p\|_{\D^n,\infty}.$
\end{thm} 
Throughout this paper, the vectors are assumed to be row vectors unless specified otherwise. 
The following series of lemmas are 
the key ingredients of this paper.
\begin{lemm}\label{UpperQuantity}
For any symmetric $n\times n$ matrix $A=\big(\!\!\big(a_{jk}\big)\!\!\big),$ we have the equality:
\[\sup_{\substack{\|x_j\|_{\ell^{2}}\leq 1}}\Big|\sum_{j,k=1}^{n}a_{jk}\langle x_j, x_k\rangle\Big| = 
\sup_{\substack{\|R_j\|_{\ell^{2}_{\mathbb R}}\leq 1}}\Big|\sum_{j,k=1}^{n}a_{jk}\langle R_j, R_k\rangle\Big|.\]
\end{lemm}
\begin{proof}
Fix $x_j\in \mathbb C^m$ with $\|x_j\|_{\ell^{2}(m)}\leq 1$ for $j=1,\ldots,n.$ 
Define $R_j=\big(\frac{\overline{x}_j+x_j}{2},i\frac{\overline{x}_j-x_j}{2}\big)$ for $j=1,\ldots, n.$ We see that 
\begin{eqnarray*}
\langle R_j,R_k \rangle &=& \sum\limits_{p=1}^{m}\Big(\frac{\overline{x}^{(p)}_j+x^{(p)}_{j}}{2}\Big)\Big(\frac{\overline{x}^{(p)}_k+x^{(p)}_{k}}{2}\Big)\\
&&+\, i^2\sum\limits_{p=1}^{m}\Big(\frac{
\overline{x}^{(p)}_j-x^{(p)}_{j}}{2}\Big)\Big(\frac{\overline{x}^{(p)}_k-x^{(p)}_{k}}{2}\Big)\\
&=& \sum\limits_{p=1}^{m}(\Re x_j^{(p)} \Re x_k^{(p)} + \Im x_j^{(p)} \Im x_k^{(p)})\\
&=& \Re \Big(\sum\limits_{p=1}^{m}x_j^{(p)}\overline{x}^{(p)}_j\Big)\\
&=& \frac{\langle x_j, x_k\rangle + \langle x_k,x_j \rangle}{2},   
\end{eqnarray*} 
where $\Re z$ and $\Im z$ denote the real and imaginary part of $z$ respectively. 
In particular, we get that $\|R_j\|_{\ell^2_{\R}(2m)}=\|x_j\|_{\ell^2(m)}$ for each $j=1,\ldots,n.$
Since $A$ is a symmetric matrix therefore 
\begin{eqnarray*}
\sum_{j,k=1}^{n} a_{jk}\langle x_j ,x_k \rangle &=& \sum_{j,k=1}^{n} a_{jk}\langle x_k ,x_j \rangle\\
&=& \sum_{j,k=1}^{n} a_{jk}\frac{\langle x_j ,x_k \rangle + \langle x_k ,x_j \rangle}{2}\\
&=& \sum_{j,k=1}^{n} a_{jk}\langle R_j ,R_k \rangle. 
\end{eqnarray*}
This shows that for each $m\in\N$ and symmetric matrix $A,$ one gets the following identity 
\[\sup_{\|x_j\|_{\ell^2(m)}\leq 1}\Big|\sum a_{jk}\langle x_j ,x_k \rangle\Big|=\sup_{\|R_j\|_{\ell^2_{\R}(2m)}\leq 1}\Big|\sum a_{jk}\langle R_j ,R_k \rangle\Big|.\]
This completes the proof.
\end{proof}

\begin{rema}\label{S(A)}
For any matrix $A,$ 
one can associate a symmetric matrix $S(A)=(A+A^{\rm t})/{2},$ 
which has the property 
$\|p_{\!_{A}}\|_{\D^n,\infty}=\|p_{\!_{S(A)}}\|_{\D^n,\infty}.$ 
Moreover, if $A$ 
is a non-negative definite matrix then $S(A)$ is a real 
non-negative definite matrix.  
\end{rema}

\begin{lemm}\label{LowerQuantity}
Suppose $A=\big(\!\!\big(a_{jk}\big)\!\!\big)$ is a non-negative definite  $n\times n$ matrix. Then 
\[\sup_{z_j\in \T}\Big|\sum_{j,k=1}^{n}a_{jk}z_jz_k\Big|=\sup_{s_j=\pm 1}\sum_{j,k=1}^{n}a_{jk}s_js_k.\]
\end{lemm}
\begin{proof}
Suppose $A$ is a $n\times n$ complex non-negative definite matrix. 
From Remark \ref{S(A)}, 
we know that $S(A)$  
is a real non-negative definite matrix with $\|p_{\!_{A}}\|_{\D^n,\infty}=\|p_{\!_{S(A)}}\|_{\D^n,\infty}.$    
Thus, to prove this lemma, it suffices to work with real non-negative definite matrices only.
 
Let $(z_1^{0},\ldots, z_n^{0})$ be a maximizing vector for 
$\|p_{\!_{A}}\|_{\D^n,\infty}$ 
i.e.  $(z_1^{0},\ldots, z_n^{0})$ satisfies 
\[\sup_{z_j\in \T}\Big|\sum_{j,k=1}^{n}a_{jk}z_jz_k\Big|= \Big|\sum_{j,k=1}^{n}a_{jk}z_j^{0}z_k^{0}\Big|.\]
Define $\tilde{z}_j=e^{-i\phi/2}z_j^{0}$ for $j=1,\ldots, n,$ 
where $\phi=\arg\big(\sum_{j,k=1}^{n}a_{jk}z_j^{0}z_k^{0}\big).$ 
Then, we have the following  
\[\sum_{j,k=1}^{n}a_{jk}\tilde{z}_j\tilde{z}_k=\Big|\sum_{j,k=1}^{n}a_{jk}z_j^{0}z_k^{0}\Big|.\]
Therefore 
\[\sum_{j,k=1}^{n}a_{jk}\tilde{z}_j\tilde{z}_k+
\overline{\sum_{j,k=1}^{n}a_{jk}\tilde{z}_j\tilde{z}_k}=
2\Big|\sum_{j,k=1}^{n}a_{jk}z_j^{0}z_k^{0}\Big|\]
and hence one concludes the following identity
\[\sup_{z_j\in \T}\Big|\sum_{j,k=1}^{n}a_{jk}z_jz_k+
\overline{\sum_{j,k=1}^{n}a_{jk}z_jz_k}\Big|=
2\sup_{z_j\in \T}\Big|\sum_{j,k=1}^{n}a_{jk}z_jz_k\Big|.\]
Since $A$ is real non-negative definite matrix therefore 
we observe the following
\begin{eqnarray*}
\frac{1}{2}\sup_{z_j\in \T}\Big|\sum_{j,k=1}^{n}a_{jk}z_j z_k &+& \overline{\sum_{j,k=1}^{n}a_{jk}z_jz_k}\Big|=\sup_{\theta_j\in \mathbb{R}}\Big|\sum_{j,k=1}^{n}a_{jk}cos(\theta_j+ \theta_k)\Big|\\
&=&\sup_{\theta_j\in \R}\Big|\sum_{j,k=1}^{n}a_{jk}(cos\theta_j cos\theta_k-sin\theta_j sin\theta_k)\Big|\\
&=& \sum_{j,k=1}^{n}a_{jk}cos\delta_j cos\delta_k-\sum_{j,k=1}^{n}a_{jk} sin\delta_j sin\delta_k\\
&\leq & \sum_{j,k=1}^{n}a_{jk}cos\delta_j cos\delta_k\\
&\leq & \sup_{s_j=\pm 1}\sum_{j,k=1}^{n}a_{jk}s_js_k, 
\end{eqnarray*}
where $\delta_j,\, j=1,\ldots, n,$ are chosen such that 
$\sum_{j,k=1}^{n}a_{jk}cos(\theta_j+ \theta_k)$ 
is positive and attains maximum in modulus at 
$\theta_j=\delta_j$ for $j=1,\ldots,n.$  
The last inequality in above computation can be deduced from the fact that 
``For any convex subset $\Omega$ of $\mathbb{R}^n$ and for any non-negative definite matrix $A$, the function $f:\Omega\to\mathbb{R}$ defined by $f(x)=\langle Ax,x\rangle$ is convex" (see \cite[Corollary 3.9.5]{Hessian}).  
Thus we get the identity 
\[\sup_{z_j\in \T}\big|\sum_{j,k=1}^{n}a_{jk}z_jz_k\big|=\sup_{s_j=\pm 1}\sum_{j,k=1}^{n}a_{jk}s_js_k.\]
This proves the claim.
\end{proof}

We now prove the main theorem as a corollary of Lemma \ref{UpperQuantity} and Lemma \ref{LowerQuantity}.
For the proof, it would be convenient to introduce, what we call, Varopoulos operators. 

Let $\h$ be a separable Hilbert space and 
$\{e_j\}_{j\in\mathbb{N}}$ be an orthonormal basis for $\h$. 
For any $x\in\h$, 
define $x^{\sharp}:\h\to\mathbb C$ by $x^{\sharp}(y)=\sum_{j}x_jy_j,$ 
where $x=\sum_j x_je_j$ and $y=\sum_j y_je_j$. 
For $x,y\in\h$, we set $\left[x^{\sharp},y\right]=x^{\sharp}(y).$ 
Then $\h^{\sharp}:=\left\{x^{\sharp}:x\in\h\right\}$ 
is a Hilbert space when equipped with the operator norm.
Since the map $\phi:\h\to\h^{\sharp}$ defined by 
$\phi(x)=x^{\sharp}$ is a linear onto isometry,  
therefore $\h^{\sharp}$ is linearly 
(as opposed to the usual anti-linear identification) 
isometrically isomorphic to $\h$. 
The following definition is taken from the Ph.D. thesis of the first named author of this paper \cite{GR} submitted to the Indian Institute of Science in 2015.

\begin{defi}[Varopoulos Operator]
	Let $\mathbb{H}$ be a separable Hilbert space. 
	For $x,y\in\mathbb{H}$, 
	define the Varopoulos operator $T_{x,y}:\mathbb C \oplus \h \oplus \mathbb C \to \mathbb C \oplus \h \oplus \mathbb C$, corresponding to the pair $(x,y)$, to be the linear transformation with the matrix representation:
	\[T_{x,y}=
	\left(
	\begin{array}{ccc}
		0 & x^{\sharp} & 0\\
		0 & 0 & y\\
		0 & 0 & 0\\
	\end{array}
	\right).\]
\end{defi}
Notice that for any pairs $(x_1,y_1)$ and $(x_2,y_2)$ in $\mathbb{H}\oplus \mathbb{H},$ the corresponding Varopoulos operators $T_{x_1,y_1}$ and $T_{x_1,y_1}$ commute if and only if $[x_1^{\sharp},y_2]=[x_2^{\sharp},y_1].$ 
If $x=y,$ then we set $T_x:=T_{x,x}.$ Since for any $x_1,x_2\in\mathbb{H},$ we have $[x_1^\sharp,x_2]=[x_2^\sharp,x_1],$ the corresponding Varopoulos operators $T_{x_1}$ and $T_{x_2}$ commute.

\begin{thm}\label{MainTheorem}
$C_{2}^{+}=\pi/2.$
\end{thm}
\begin{proof}
Suppose $A$ is $n\times n$ non-negative definite matrix 
and suppose that $(T_1,\ldots,T_n)$ is a tuple in $\C_n,$ then,
\[p_{\!_A}(T_1,\ldots,T_n)=p_{\!_{S(A)}}(T_1,\ldots,T_n).\]
If $b_{jk}$ denotes the $(j,k)$ entry of 
the matrix $S(A)$ then for every $x,y\in\mathbb{H},$ we have the following

\begin{eqnarray*}
\sup_{\|x\|\leq 1,\,\|y\|\leq 1}| \langle p_{\!_{S(A)}}(T_1,\ldots,T_n)x,y \rangle | &=& \sup_{\|x\|\leq 1,\,\|y\|\leq 1}\big| \sum_{j,k=1}^{n} b_{jk}  \langle T_j x, T_k^* y\rangle \big|\\
& \leq & \sup_{\substack{\|x_j\|\leq 1,\,\|y_k\|\leq 1}}\big| \sum_{j,k=1}^{n} b_{jk}\langle x_j, y_k\rangle \big|\\
& \leq & \sup_{\substack{\|x_j\|\leq 1}}\big| \sum_{j,k=1}^{n} b_{jk}\langle x_j, x_k\rangle \big|. 
\end{eqnarray*}
The last inequality is explained by the following computation, which can essentially be found in \cite{AN}. 
Since $(\!(b_{jk})\!)$ is a non-negative definite matrix 
therefore 
there exist $V_1,\ldots,V_n\in\mathbb{C}^n$ such that $b_{jk}=\langle V_j, V_k\rangle$ 
for each $j,k=1,\ldots,n.$ 
\begin{eqnarray*}
\sup_{\substack{\|x_j\|\leq 1\\ \|y_k\|\leq 1}}\,\Big| \sum_{j,k=1}^{n} b_{jk}  \langle x_j, y_k\rangle \Big|
&=& \sup_{\substack{\|x_j\|\leq 1\\ \|y_k\|\leq 1}}\,\Big| \sum_{p=1}^{m}\big\langle\sum_{j=1}^{n}x_{jp}V_j,\sum_{k=1}^{n}{y}_{kp}V_k\big\rangle\Big|\\
&\leq & \sup_{\substack{\|x_j\|\leq 1\\ \|y_k\|\leq 1}} \sum_{p=1}^{m}\Big|\big\langle\sum_{j=1}^{n}x_{jp}V_j,\sum_{k=1}^{n}{y}_{kp}V_k\big\rangle\Big|\\
&\leq & \sup_{\substack{\|x_j\|\leq 1}} \Big(\sum_{p=1}^{m}\big\|\sum_{j=1}^{n}x_{jp}V_j\big\|^2\Big)^{1/2}\\
&&\sup_{\|y_k\|\leq 1} \Big(\sum_{p=1}^{m}\big\|\sum_{k=1}^{n}{y}_{kp}V_k\big\|^2\Big)^{1/2}\\
&=& \sup_{\|x_j\|\leq 1}\sum_{p=1}^{m}\big\|\sum_{j=1}^{n}x_{jp}V_j\big\|^2\\
&=& \sup_{\|x_j\|\leq 1} \sum_{p=1}^{m}\big\langle\sum_{j=1}^{n}x_{jp}V_j,\sum_{k=1}^{n}{x}_{kp}V_k\big\rangle\\
&=& \sup_{\|x_j\|\leq 1}\big| \sum_{j,k=1}^{n} b_{jk}  \langle x_j, x_k\rangle \big|.\\
\end{eqnarray*}
Therefore, we get the following inequality
\[\sup_{\|x\|\leq 1,\,\|y\|\leq 1}| \langle p_{\!_{S(A)}}(T_1,\ldots,T_n)x,y \rangle |\leq \sup_{\|x_j\|\leq 1}\big| \sum_{j,k=1}^{n} b_{jk}  \langle x_j, x_k\rangle \big|.\] 
Also note that 
\begin{eqnarray*}
\|(\!(b_{jk})\!)\|_{\ell^\infty_\mathbb{R}(n)\to \ell^1_\mathbb{R}(n)}&=& \sup_{s_j,t_k \in [-1,1]}\Big| \sum_{j,k=1}^{n} b_{jk}  s_j t_k \Big|\\
&=& \sup_{s_j,t_k \in [-1,1]}\Big| \big\langle\sum_{j=1}^{n}s_{j}V_j,\sum_{k=1}^{n}{t}_{k}V_k\big\rangle\Big|\\
&\leq &  \sup_{s_j,t_k \in [-1,1]} \Big(\big\|\sum_{j=1}^{n}s_{j}V_j\big\|^2\Big)^{1/2} \Big(\big\|\sum_{k=1}^{n}{t}_{k}V_k\big\|^2\Big)^{1/2}\\
&=&  \sup_{s_j \in [-1,1]}\big\|\sum_{j=1}^{n}s_{j}V_j\big\|^2\\
&=& \sup_{s_j \in [-1,1]}\Big| \sum_{j,k=1}^{n} b_{jk}  s_j s_k \Big|.
\end{eqnarray*}
Thus we get the identity 
$\|(\!(b_{jk})\!)\|_{\ell^\infty_\mathbb{R}(n)\to \ell^1_\mathbb{R}(n)}=\|p_{\!_{S(A)}}\|_{[-1,1]^n,\infty}.$
By this identity and Lemma \ref{LowerQuantity} 
we get the following 
\begin{eqnarray}\label{preVIandGI}
\sup_{\substack{A\in M_n^+(\mathbb{C})\setminus \{0\}}}&&\!\!\!\!\!\!\!\!\!\!\!\frac{\sup_{\boldsymbol T\in \mathscr{C}_n}\|p_{\!_A}(T_1,\ldots,T_n)\|}{\|p_{\!_A}\|_{\D^n,\infty}}\\ 
&&\!\!\!\!\!\!\!\!\!\!\!\!\!\!\!\!\!\!\!\!\!\!\leq \sup_{\substack{B\in M_n^+(\mathbb{R})\setminus \{0\}}}\frac{\sup_{\substack{\|x_j\|_{\ell^2}\leq 1}}\big| \sum_{j,k=1}^{n} b_{jk}\langle x_j, x_k\rangle \big|}{\|B\|_{\ell^{\infty}_{\R}(n)\to \ell^1_{\R}(n)}}\nonumber. 
\end{eqnarray}
Using the definition of $C_{2}^{+},$ Lemma \ref{UpperQuantity} and the inequality  \eqref{preVIandGI}, we get the following 
\begin{eqnarray*}
C_{2}^{+} \leq \!\!\!\!\!\sup_{\substack{B\in M_n^+(\mathbb{R})\setminus \{0\}}}\frac{\sup_{\substack{\|R_j\|_{\ell^2_{\R}}\leq 1}}\big| \sum_{j,k=1}^{n} b_{jk}\langle R_j, R_k\rangle \big|}{\|B\|_{\ell^{\infty}_{\R}(n)\to \ell^1_{\R}(n)}}=K_G^+(\R)=\pi/2.
\end{eqnarray*}

Fix an $n\times n$ non-negative definite matrix $A$ 
and $x_j\in \mathbb C^m,$ $j=1,\ldots,n,$ for some $m\in \mathbb{N}.$  
Define the Varopoulos operator $T_j\,(=T_{R_j})$ corresponding to the vector $R_j\in\mathbb{R}^{2m}\,(\subset \mathbb{C}^{2m}),$ where $R_j=(\frac{\overline{x}_j+x_j}{2},i\frac{\overline{x}_j-x_j}{2})$ 
for $j=1,\ldots,n.$ 
Then, taking the form of $p_{\!_A}(T_1,\ldots,T_n)$ into account, we get 
\[\|p_{\!_A}(T_1,\ldots,T_n)\|=\sum_{j,k=1}^{n} a_{jk}\langle R_j, R_k\rangle=\sum_{j,k=1}^{n} \big(\frac{a_{jk}+a_{kj}}{2}\big)\langle R_j, R_k\rangle.\]  
We use Lemma \ref{LowerQuantity} to subsequently obtain
\begin{eqnarray*}
\sup_{\substack{A\in M_n^+(\mathbb{C})\setminus \{0\}}}&&\!\!\!\!\!\!\!\!\!\!\!\!\!\!\!\frac{\sup_{\boldsymbol T\in \mathscr{C}_n}\|p_{\!_A}(T_1,\ldots,T_n)\|}{\ \ \ \ \ \ \ \ \ \, \|p_{\!_A}\|_{\D^n,\infty}}\\
&&\geq \sup_{\substack{B\in M_n^+(\mathbb{R})\setminus \{0\}}}\frac{\sup_{\substack{\|R_j\|_{\ell^2_{\R}}\leq 1}}\big| \sum_{j,k=1}^{n} b_{jk}\langle R_j, R_k\rangle \big|}{\|B\|_{\ell^{\infty}_{\R}(n)\to \ell^1_{\R}(n)}}.
\end{eqnarray*}
This proves the theorem. 
\end{proof}

Theorem \ref{MainTheorem} shows that if 
we restrict ourselves to the set of homogeneous polynomials of degree two 
coming from the real non-negative definite matrices, then we obtain an analogous inequality to that of  Varopoulos, where the factor $2$ on the right disappears  
and the constant $K_G^{\mathbb C}$ gets replaced by $K_G^{+}(\R),$ that is, 
\[\sup_{n, p_A}\|p_{\!_A}(T_1,\ldots,T_n)\|=K_G^{+}(\R),\]
where supremum is taken over all homogeneous polynomials $p_{\!_A}$ of supremum norm at most $1$ 
with $A$ being a non-negative definite matrix, the tuple 
$(T_1,\ldots,T_n)$ being arbitrary in $\mathscr{C}_n$ and $n\in\mathbb{N}.$ 
The next corollary shows that $\lim_{n\to\infty} C_2(n)/K^\mathbb C_G > 1.$ 
This, in turn, answers a long--standing question of Varopoulos, raised in \cite{V2}, in negative. 

\begin{coro} \label{Ratio > 1}
For some $\epsilon>0,$ we have the inequality
$$\lim_{n\to \infty}C_2(n)\geq (1+\epsilon) K_{G}^{\mathbb C}.$$
\end{coro}
\begin{proof}
We know that $C_2(n)\geq C_{2}^{+}(n)$ for each $n\in \N$ 
and $K_G^\mathbb C\leq 1.4049$ 
therefore we have the following 
\[\lim_{n\to \infty}C_2(n) \geq \lim_{n\to \infty}C_{2}^{+}(n)=\pi/2> K_{G}^{\mathbb C}.\]
To complete the proof one can take $\epsilon=0.118.$
\end{proof}

A variant of Corollary \ref{Ratio > 1} on $L^p$ - spaces can be found in  \cite{SRK}.

In the next theorem, we compute a lower bound for the norm of 
$\mathscr{A}_n$ as $n$ tends to infinity. 
This improves a bound obtained earlier in an unpublished paper of 
Holbrook and Schoch where they had proved  that 
$\lim_{n\to \infty}\|\mathscr{A}_n\|\geq 1.2323.$
\begin{thm}
$\lim_{n\to \infty}\|\mathscr{A}_n\|\geq \frac{\pi^2}{8}.$
\end{thm}
\begin{proof}
Fix a natural number $l.$ 
Since $\lim_{n\to \infty}C_{2}^{+}(n)=\pi/ 2,$ 
there exist $n\in\mathbb{N}$ and a real non-negative definite matrix $A_l=(\!(a^{(l)}_{jk})\!)$ of size $n$  such that 
\[\frac{\sup_{\boldsymbol T\in \mathscr{C}_n}\|\sum a^{(l)}_{jk}T_jT_k\|}{\sup_{z_j\in\T}|\sum a^{(l)}_{jk}z_jz_k|}\geq \pi/2 -1/l.\]
By Lemma \ref{UpperQuantity} and 
the fact that $K_G^+(\mathbb C)=4/\pi,$ we get the following for the matrix $A_l$
\begin{eqnarray*}
4/\pi &\geq & \frac{\sup_{\substack{\|x_j\|_{\ell^{2}}}\leq 1}\sum_{j,k=1}^{n}a^{(l)}_{jk}\langle x_j,x_k \rangle}{\|A_l\|_{\ell^\infty(n)\to \ell^1(n)}}\\
&=&\frac{\sup_{\substack{\|R_j\|_{\ell^{2}_{\R}}\leq 1}}\sum_{j,k=1}^{n}a^{(l)}_{jk}\langle R_j,R_k \rangle}{\|A_l\|_{\ell^\infty(n)\to \ell^1(n)}}\\
&=& \frac{\sup_{\boldsymbol T\in\C_n}\|\sum a^{(l)}_{jk}T_jT_k\|}{\sup_{z_j\in\T}|\sum a^{(l)}_{jk}z_jz_k|}\cdot 
\frac{\sup_{z_j\in\T}\|\sum a^{(l)}_{jk}z_jz_k\|}{\|A_l\|_{\ell^\infty(n)\to \ell^1(n)}}\\
&\geq & (\pi/2-1/l)\frac{\!\!\!\!\!\!\|p_{\!_{A_l}}\|_{\D^n,\infty}}{\ \ \|A_l\|_{\ell^\infty(n)\to \ell^1(n)}}.
\end{eqnarray*}
Rewriting the inequality appearing above, 
we see that, for each $l\in \mathbb{N},$ 
there exists a symmetric matrix $A_l$ such that 
for the corresponding polynomial $p_{\!_{A_l}},$ 
we have the following estimate    
\[\frac{\|A_l\|_{\ell^\infty(n)\to \ell^1(n)}}{\!\!\!\!\!\!\|p_{\!_{A_l}}\|_{\D^n,\infty}}\geq \frac{(\pi/2-1/l)}{4/\pi}.\]
Taking supremum over all the natural numbers $l$ on both the sides, 
we get the following inequality
\[\sup_{l\in\N}\frac{\|A_l\|_{\ell^\infty(n)\to \ell^1(n)}}{\!\!\!\!\!\!\|p_{\!_{A_l}}\|_{\D^n,\infty}}\geq \frac{\pi^2}{8}\approx 1.2337.\]
The result follows immediately from here.
\end{proof}

\section{\VK type examples and the constant \texorpdfstring{$C_2(n)$}{TEXT}}\label{MaximizingLemma} 
In this section, we focus on the constant $C_2(n)$ in more detail. 
We discuss the asymptotic behaviour of $C_2(n)$ and 
exhibit an explicit class of examples of \VK type. 
For this, we mainly rely on a construction which appeared in \cite{FJ}. 
In this case, our main tool is a very general maximizing lemma 
which can also be of independent interest. 
For $n=3$, we successfully construct 
a very wide class of examples like \VK for which 
the von Neumann inequality fails and 
show that the \VK example is extremal on the class of certain 
$3\times 3$ symmetric matrices including symmetric sign matrices.
 
In \cite{FJ}, the authors produced an explicit set of 
real non-negative definite matrices for which 
the positive Grothendieck constant goes up to $1.5$. 
We briefly describe their matrices as follows:

For $l=k(k-1)$, define $F_k=\{v_1,\dots,v_{k(k-1)}\}$, 
the set of all $k$-dimensional vectors with two non-zero components, 
either $1$ and $1$ or $1$ and $-1,$ appearing in that order. 
Define a real $l\times l$ non-negative definite matrix 
$A_k=(\!(a_{ij})\!)_{1\leq i,j\leq l}$ 
as $a_{ij}=\langle v_i,v_j\rangle$. 
In \cite{FJ}, the authors showed that 
$$\frac{\sup_{\|X_i\|_2=1}\sum_{i,j=1}^la_{ij}\langle X_i,X_j\rangle}{{\sup_{\omega_i\in\{1,-1\}}}\sum_{i,j=1}^la_{ij}\omega_i\omega_j}=\frac{3k-3}{2k-1}.$$

Thus, in view of Lemma \ref{LowerQuantity}, 
we get a large class of \VK like examples for which the von Neumann inequality fails and $C_2(k(k-1))\geq\frac{3k-3}{2k-1}$. 
Notice that $\frac{3k-3}{2k-1}$ is an increasing function in $k$ and increases to $\frac{3}{2}$ as $k$ tends to infinity. 
Hence for this explicit class of matrices, we get the lower bound of $\lim_{n\to\infty} C_2(n)$ to be equal to $\frac{3}{2}.$ 
Though we already got the lower bound of $\lim_{n\to\infty} C_2(n)$ to be equal to $\frac{\pi}{2},$ it will be interesting to get an estimate of $C_2(n)$ as a function of $n$.

Motivated by the example of Varopoulos and Kaijser, provided in \cite{V1}, we develop the following maximizing lemma which enables us to compute the supremum norm of polynomials. 
\begin{lemm}[Maximizing Lemma]\label{ML}
Let $\Omega$ be a bounded domain in $\mathbb{R}^n$. Suppose a function $F=(f_1,\dots,f_m):\overline{\Omega}\subseteq\mathbb{R}^n\to \mathbb{C}^m$ 
is a continuously differentiable and bounded function with 
$f_j$ non-vanishing for $j=1,\ldots,m$. 
Then we have the following 
$$\big\{x\in\overline{\Omega}:\|F(x)\|_{\ell^\infty(m)}=\|F\|_{\Omega,\infty}\big\}\subseteq\bigcup_{j=1}^m\big\{x\in\overline{\Omega}:\dim_{\mathbb{R}}\bigvee_{k=1}^n\frac{\partial f_j}{\partial x_k}(x)\leq 1\big\},$$
where $\bigvee\{v_1,\ldots,v_n\}$ denotes the subspace 
spanned by the vectors $v_1,\ldots,v_n$ and 
$\dim_{\mathbb{R}}$ denotes the dimension of the space over the field of real numbers $\mathbb{R}.$ 
\end{lemm}
\begin{proof}
We notice the following identity 
\begin{equation*}
\{x\in\overline{\Omega}:\|F(x)\|_{\ell^\infty(m)}=\|F\|_{\Omega,\infty}\}\subseteq	\bigcup_{j=1}^m\big\{x\in\overline{\Omega}:|f_j(x)|=\|f_j\|_{\Omega,\infty}\big\}.
\end{equation*}
For $f_j:\overline{\Omega}\subseteq\mathbb{R}^n\to \mathbb{C},$ we have 
\begin{eqnarray}\label{Modulusfj}
|f_j|^2=(\Re f_j)^2+(\Im f_j)^2.
\end{eqnarray}
Differentiating \eqref{Modulusfj} with respect to $x_k$ on both the sides, 
we get the following
\begin{equation}\label{PR}
\frac{\partial}{\partial x_k}|f_j|^2=2\Re f_j\Re\frac{\partial f_j}{\partial x_k}+2\Im f_j\Im\frac{\partial f_j}{\partial x_k}.
\end{equation}
By \eqref{PR} and the fact that, at the point of maximum of $|f_j|,$ 
all the partial derivatives of $|f_j|^2$ are zero, we obtain
\begin{equation}\label{pr1}
\langle(\Re f_j,\Im f_j),\big(\Re\frac{\partial f_j}{\partial x_k},\Im\frac{\partial f_j}{\partial x_k}\big)\rangle=0,\ \forall\ 1\leq k\leq n.
\end{equation}
From Equation \eqref{pr1}, 
we observe that, at the point of maximum of $|f_j|^2$, 
the two dimensional vectors 
$\big(\Re\frac{\partial f_j}{\partial x_k},\Im\frac{\partial f_j}{\partial x_k}\big),$ $1\leq k\leq n,$ lie on a line passing through the origin in $\mathbb{R}^2.$
Therefore,
\begin{equation}
\big\{x\in\overline{\Omega}:|f_j(x)|=\|f_j\|_{\Omega,\infty}\big\}\subseteq\big\{x\in\overline{\Omega}:\dim_\mathbb{R}\bigvee_{k=1}^n\frac{\partial f_j}{\partial x_k}(x)\leq 1\big\}.
\end{equation} 
This completes the proof.
\end{proof}
\begin{rema}
To disprove 
the von Neumann inequality in three variables,  
Varopoulos and  Kaijser \cite{V1} considered an 
explicit homogeneous polynomial  
of degree two in three variables. 
While the computation of the supremum norm 
of this particular polynomial is briefly indicated in their paper, 
we indicate below,  using Lemma \ref{ML},
how to compute the supremum norm of the Varopoulos--Kaijser polynomial. 
Of course, this recipe applies to the entire class of Varopoulos polynomials.
\end{rema}
 
The Varopoulos-Kaijser polynomial is the following homogeneous polynomial of degree two 	
$$p(z_1,z_2,z_3)=z_1^2+z_2^2+z_3^2-2z_1z_2-2z_2z_3-2z_3z_1.$$
Without loss of generality, we can take supremum over  
$\{(z_1,z_2,z_3)=(1,e^{i\theta},e^{i\phi}): \theta,\phi\in \mathbb{R}\}$ 
in the expression of $p(z_1,z_2,z_3)$ to compute 
the quantity $\|p\|_{\D^3,\infty}$. If we consider the function   
$g(\theta,\phi)=1+e^{2i\theta}+e^{2i\phi}-2e^{i\theta}-2e^{i\phi}-2e^{i(\theta+\phi)}$
 then $\|p\|_{\D^3,\infty}=\sup_{\theta,\phi\in \mathbb{R}}|g(\theta,\phi)|.$
 Taking partial derivatives with respect to $\theta$ and $\phi$ of $g,$ 
we get the following expressions 
\begin{eqnarray*}
\frac{\partial g}{\partial\theta}(\theta,\phi)=2ie^{2i\theta}-2ie^{i\theta}-2ie^{i(\theta+\phi)}\\\frac{\partial g}{\partial\phi}(\theta,\phi)=2ie^{2i\phi}-2ie^{i\phi}-2ie^{i(\theta+\phi)}.
\end{eqnarray*}
Applying Lemma \ref{ML}, we obtain that, 
at the point of maximum of $|g|,$ 
the vectors 
$\frac{1}{2i}\frac{\partial g}{\partial\theta}$, $\frac{1}{2i}\frac{\partial g}{\partial\phi}$ and 
$g(\theta,\phi)-\frac{1}{2i}(\frac{\partial g}{\partial\theta}+\frac{\partial g}{\partial\phi})$ 
lie on a line passing through the origin in $\mathbb{R}^2.$
Note that 
\begin{eqnarray*}
g(\theta,\phi)-\frac{1}{2i}\Big(\frac{\partial g}{\partial\theta}+\frac{\partial g}{\partial\phi}\Big)=1-e^{i\theta}-e^{i\phi}\\
\frac{1}{2i}\frac{\partial g}{\partial\theta}-\frac{1}{2i}\frac{\partial g}{\partial\phi}=(e^{i\theta}-e^{i\phi})(e^{i\theta}+e^{i\phi}-1).
\end{eqnarray*}
Therefore, at the point of maximum, $1-e^{i\theta}-e^{i\phi}$ 
and $(e^{i\theta}-e^{i\phi})(e^{i\theta}+e^{i\phi}-1)$ 
lie on a line passing through the origin in $\mathbb{R}^2$. 
Since $1-e^{i\theta}-e^{i\phi}$ and $e^{i\theta}+e^{i\phi}-1$ 
are always collinear, 
one must have $1-e^{i\theta}-e^{i\phi}=0$ or 
$\arg(e^{i\theta}-e^{i\phi})\in\{0,\pi\}$. 
From this, it can be concluded that $\|p\|_{\D^3,\infty}=5.$ 

We would like to bring the attention of the reader to \cite{HJ} where the computation of the supremum norm of this polynomial has also been done.
  
\subsection{Extremal behaviour of \VK example:}
In this subsection, we show that the \VK example is extremal, 
in a certain sense,  
if we restrict ourselves to a class of symmetric   
$3\times 3$ matrices which includes the symmetric sign matrices. 
Sign matrices are the matrices of which 
each entry is either $1$ or $-1$. 
\VK example gives a lower bound for the quantity $C_2(\boldsymbol \delta_3)$
and using extreme point method, 
we establish an upper bound for the same. 
For this we need the following definitions.
\begin{defi}[Correlation Matrix]
A correlation matrix is a complex non-negative definite matrix 
whose all diagonal elements are equal to $1$. 
We denote the set of all $n\times n$ 
correlation matrices by $\mathscr{C}(n)$.
\end{defi} 

For every natural number $n,$ 
define the set $\boldsymbol{\delta}_n$ by
\[ \boldsymbol{\delta}_n=\big\{(T_{x_1},\ldots,T_{x_n}):x_j\in \ell^2_\mathbb{R} \mbox{ with }\|x_j\|\leq 1\mbox{ for }j=1,\ldots,n \big\}.\] 
Define 
$C_2(\boldsymbol{\delta}_n):=\sup\big\{ \|p(\boldsymbol T)\|:\boldsymbol T\in\boldsymbol\delta_n,\,p\in\mathcal{P}_2^s(n)\mbox{ with }\|p\|_{\mathbb{D}^n,\infty}\leq 1\big\}.$ 
For the rest of this section, we consider only the  
tuples of commuting and contractive Varopoulos operators in $\boldsymbol{\delta}_n$.
From the definition, it follows that 
$C_2(\boldsymbol{\delta}_n)\leq C_2(n).$  
We prove the following bound on $C_2(\boldsymbol \delta_3).$ 

\begin{rema}\label{RestrictedVaropoulos}
Given $x_1,\ldots,x_n\in \ell^2,$ 
we can define the vector 
$R_j=\big(\frac{\overline{x}_j+x_j}{2},i\frac{\overline{x}_j-x_j}{2}\big)$ 
for $j=1,\ldots, n.$ 
As noticed in Lemma \ref{UpperQuantity}, 
we know that 
$\langle R_j,R_k\rangle=\frac{\langle x_j, x_k\rangle + \langle x_k,x_j \rangle}{2}$. 
Hence for any degree two homogeneous polynomial 
$p(z_1,\ldots,z_n)=\sum_{j,k=1}^{n}a_{jk}z_jz_k$, 
where $(\!(a_{jk})\!)$ is a symmetric matrix, we get that 
\[\sup_{\boldsymbol T\in \boldsymbol\delta_n}\|p(\boldsymbol T)\|=\sup_{\|x_j\|_{\ell^2}\leq 1}\big|\sum_{j,k=1}^{n}a_{jk}\langle x_j,x_k\rangle\big|.\] 
\end{rema}
\begin{thm}\label{HS}
$1.2\leq C_2(\boldsymbol \delta_3)\leq\frac{3\sqrt{3}}{4}.$
\end{thm}
\begin{proof} 
The fact that $C_2(\boldsymbol \delta_3)\geq 1.2$ follows from \cite{HJ}.

Given a complex $n\times n$ matrix $A,$ we define the following quantity 
$$\beta(A):=\sup_{B\in\mathscr{C}(n)}|\langle A,B\rangle|.$$ 
Every correlation matrix $B$ can be written as 
$(\!(\langle x_i,x_j\rangle)\!)$ for some unit vectors 
$x_i,$ $1\leq i\leq n,$ and vice versa. 
Let $U$ denote the unit ball of $\mathbb{C}^s[Z_1,\ldots, Z_n]$ 
with respect to supremum norm over the polydisc $\mathbb{D}^n.$ 
Suppose $A_p$ denote the symmetric matrix 
corresponding to $p\in \mathbb{C}^s[Z_1,\ldots, Z_n]$.
Then, using Remark \ref{RestrictedVaropoulos}, we get 
\begin{eqnarray*}
\sup_{p\in U}\beta(A_p)&=&\sup_{p\in U}\sup_{\|x_i\|_{\ell^2}=1}\big|\sum_{i,j=1}^na_{ij}\langle x_i,x_j \rangle\big|\\
&=& C_2(\boldsymbol \delta_n).
\end{eqnarray*}
The map $B\mapsto \langle A,B\rangle$ is linear in $B$ 
and $\mathscr{C}(n)$ is a compact convex set, we conclude that 
\begin{eqnarray}\label{LinearInB}
\beta(A_p)=\sup_{B\in E(\mathscr{C}(n))}|\langle A_p,B\rangle|,
\end{eqnarray}
where $E(\mathscr{C}(n))$ is the set of all extreme points of $\mathscr{C}(n)$.
Since all the elements of $E(\mathscr{C}(n))$ 
have rank less than or equal to $\sqrt{n}$ (\cite{LCB}) 
therefore when $n=3,$ 
we conclude that the extreme correlation matrices have rank one. 
If the correlation matrix 
$(\!(\langle x_i,x_j\rangle)\!)$ is of rank $1$, 
then $\bigvee\{x_i;1\leq i \leq n\}$ is one dimensional. 
Using Equation \eqref{LinearInB} and \cite{RG}, 
for $n=3$, we obtain the following 
\begin{eqnarray*}
\beta(A_p)&=&\sup_{\mid z_i\mid=1}\Big|\sum_{i,j=1}^{n}a_{ij}z_i\bar{z}_{j}\Big|\\
&\leq &\|A_p\|_{\infty\to 1}\\
&\leq &\frac{3\sqrt{3}}{4}\sup_{\mid z_i\mid=1}\Big|\sum_{i,j=1}^{n}a_{ij}z_i{z}_{j}\Big|.
\end{eqnarray*}
This completes the proof of the theorem.
\end{proof} 

\begin{rema} In an unpublished work, Holbrook and Schoch have  shown that $C_2(\boldsymbol \delta_3)\geq 1.2323$. Their method rely on explicit construction of a two degree homogeneous polynomial as in \VK (replacing the coefficients $-2$ in the \VK polynomial by something like $-2.5959$). In view of this and Theorem \ref{HS}, we have $1.2323\leq C_2(\boldsymbol \delta_3)\leq\frac{3\sqrt{3}}{4}\approx 1.2990.$
\end{rema}
The following table shows that \VK polynomial is extremal 
among the set of all symmetric sign matrices of order $3$ 
as long as the ratio 
$\|p(T_1,T_2,T_3)\|/\|p\|_{\mathbb{D}^3,\infty}$ is concerned, 
where $T_1,T_2$ and $T_3$ are commuting Varopoulos operators. 
The total number of symmetric sign matrices of order $3$ is $2^6.$ 
To compute $\|p\|_{\mathbb{D}^3,\infty}$ and $\|p(T_1,T_2,T_3)\|,$ 
without loss of generality, we can assume that 
every entry in first row and first column of $A_p,$ 
symmetric matrix corresponding to $p,$ is one. 
Since $\|p\|_{\mathbb{D}^3,\infty}$ and $\|p(T_1,T_2,T_3)\|$ 
are invariant under $SA_pS^{-1},$ 
for every permutation matrix $S$ of order $3,$ 
therefore it leaves us with the following $6$ inequivalent matrices, 
for which, we use Lemma \ref{ML} to compute the supremum norm of the polynomials.

\begin{center}

\begin{tabular}{|c|c|c|}
    \hline
    $ A_p$ & $\|p\|_{\mathbb{D}^3,\infty}$ & $\|p(T_1,T_2,T_3)\|$\\
    \hline
    $\left(
\begin{array}{rrr}
1 & 1 & 1\\
1 & -1 & 1\\
1 & 1 & -1\\
\end{array}
\right)
$ & $5$ &  $5$\\
\hline
    $\left(
\begin{array}{rrr}
1 & 1 & 1\\
1 & 1 & -1\\
1 & -1 & 1\\
\end{array}
\right)
$ & $5$ &  $6$\\
\hline
   $\left(
\begin{array}{rrr}
1 & \ \ 1 & 1\\
1 & \ \ 1 & 1\\
1 & \ \ 1 & -1\\
\end{array}
\right)
$ & $5\sqrt{2}$ &  $7$\\
\hline
$\left(
\begin{array}{rrr}
1 & 1 & 1\\
1 & 1 & -1\\
1 & -1 & -1\\
\end{array}
\right)
$ & $7$ &  $5$\\
\hline
$\left(
\begin{array}{rrr}
1 & 1 & 1\\
1 & -1 & -1\\
1 & -1 & -1\\
\end{array}
\right)
$ & $\sqrt{41}$ &  $4$\\
\hline
$\left(
\begin{array}{rrr}
1 & \ \ 1 & \ \ 1\\
1 & \ \ 1 & \ \ 1\\
1 & \ \ 1 & \ \ 1\\
\end{array}
\right)
$ & $9$ &  $9$\\
    \hline
  \end{tabular}
\end{center}
We show that \VK polynomial is also extremal among the following matrices 
\[\mathcal{S}:=\Big\{\mathscr{B}_\alpha=\left(
\begin{smallmatrix}
1 & 1 & 1\\
1 & 1 & \alpha\\
1 & \alpha & 1\\
\end{smallmatrix}
\right):\alpha\in \mathbb{R}\Big\}.\]  
For $\alpha\geq 0,$ the ratio 
$\|p(T_1,T_2,T_3)\|/\|p\|_{\mathbb{D}^3,\infty}$ 
is always $1.$ Hence we consider the case when $\alpha<0.$ 
Explicit computation shows that, 
for every symmetric matrix $\mathscr{B}_\alpha$ in $\mathcal{S},$ 
the corresponding homogeneous polynomial $p_{\!_{\mathscr{B}_\alpha}}$ of degree two 
satisfies the following 
$$\|p_{\!_{\mathscr{B}_\alpha}}\|_{{\D^3,\infty}}=\sup_{\theta,\phi\in \mathbb{R}}\big\{|1+e^{i\theta}+e^{i\phi}+e^{i\theta}(1+e^{i\theta}+\alpha e^{i\theta})+e^{i\phi}(1+e^{i\phi}+\alpha e^{i\theta})|\big\}.$$
Suppose 
$f(\theta,\phi)=|1+e^{i\theta}+e^{i\phi}+e^{i\theta}(1+e^{i\theta}+\alpha e^{i\theta})+e^{i\phi}(1+e^{i\phi}+\alpha e^{i\theta})|.$
Using Lemma \ref{ML}, 
at point of maximum of $f$, we get that 
$1+e^{i\theta}+e^{i\phi}, \ e^{i\theta}(1+e^{i\theta}+\alpha e^{i\theta})\ \text{and}\  e^{i\phi}(1+e^{i\phi}+\alpha e^{i\theta})\ \text{are collinear}.$ 
Subtracting the third vector from the second vector, we get 
$1+e^{i\theta}+e^{i\phi},(e^{i\theta}-e^{i\phi})(1+e^{i\theta}+e^{i\phi})\ \text{are collinear}.$
We break the computation essentially into the following two cases

\textbf{Case 1:}\label{Case1} If $1+e^{i\theta}+e^{i\phi}$ is zero then the maximum of $f$ is $2-2\alpha.$

\textbf{Case 2:} Suppose $1+e^{i\theta}+e^{i\phi}\neq 0.$ 
Then $1$ and $e^{i\theta}-e^{i\phi}$ are collinear 
and therefore at a point of maximum of $f,$ 
we get that $\theta=\phi$ or $\theta=\phi+\pi.$ We deal with this case in the form of following two subcases.
\begin{itemize}
\item[1.] When $\theta=\phi$ then we need to maximize 
$f(\theta,\theta)=|1+4e^{i\theta}+2(1+\alpha)e^{2i\theta}|$ 
over $\theta\in \mathbb{R}.$ 
We observe that 
$f(\theta,\theta)=(17+4(1+\alpha)^2+8(3+2\alpha)cos\theta +4(1+\alpha)cos 2\theta)^{1/2}.$ 
At critical point $\theta_0$ of $f,$ we get that 
\[cos \theta_0=-\frac{3+2\alpha}{2(1+\alpha)}\mbox{ or }sin\theta_0=0.\]
If $\alpha> -5/4$ then the only possibility is $sin\theta_0=0$ 
i.e. $e^{i\theta_0}=\pm 1.$
In this case, maximum of $f$ is either 
$|7+2\alpha|$ or $1-2\alpha.$ 
As case 1 suggests, the quantity $1-2\alpha$ can not be the maximum.
In this subcase if $\alpha$ is at most $-5/4$ then at 
$cos \theta_0=-\frac{3+2\alpha}{2(1+\alpha)},$ 
$$f(\theta_0,\theta_0)=\big(17+4(1+\alpha)^2-2\frac{(3+2\alpha)^2}{\!\!\!1+\alpha}-4(1+\alpha)\big)^{1/2}.$$  
\item[2.] When $\theta=\phi+\pi$ then maximum of $f$ is $3-2\alpha.$ 
This subcase proves the redundancy of case 1 as far as the maximum of $f$ is concerned.
\end{itemize}
 
Comparison of all the possible cases 
and explicit computation tells us that 
for $\alpha<0,$ the norm of the homogeneous polynomial 
$p_{\!_{\mathscr{B}_\alpha}}$ of degree two, 
is the following continuous function 
\[\|p_{\!_{\mathscr{B}_\alpha}}\|_{\D^3,\infty}=
\begin{cases}
7+2\alpha, & \alpha>-1\\
3-2\alpha, & \alpha\leq -1\\
\end{cases}
\]
We define the following quantity 
$M_{\mathscr{B}_\alpha}=\sup_{|z_j|=1}\big|\sum a_{jk}z_j\overline{z}_k\big|.$ 
By Remark \ref{RestrictedVaropoulos} and from \cite{LCB}, 
we see that 
$M_{\mathscr{B}_\alpha}=\sup_{\boldsymbol T\in \boldsymbol\delta_3}\|p_{\!_{\mathscr{B}_\alpha}}(\boldsymbol T)\|.$ 
Corresponding to every matrix $\mathscr{B}_\alpha$ in the class $\mathcal{S},$ we get 
\begin{eqnarray*}
M_{\mathscr{B}_\alpha}&=&\sup_{|z_j|=1}\big|3+2(\Re z_1\overline{z}_2+\alpha\Re z_2\overline{z}_3+\Re z_3\overline{z}_1)\big|\\
&=& \sup_{\theta,\phi\in\mathbb{R}}\big|3+2(cos\theta + cos\phi + \alpha cos(\theta-\phi))\big|
\end{eqnarray*}
Define the function $h(\theta,\phi)=3+2(cos\theta + cos\phi + \alpha cos(\theta-\phi)).$ 
Then the critical points of the function $h$ 
are the solutions of $sin\theta +\alpha sin(\theta-\phi)=0$ and 
$sin\phi -\alpha sin(\theta-\phi)=0.$  
Therefore the critical points $(\theta_0,\phi_0)$ 
satisfy $sin\theta_0 + \sin\phi_0=0.$ 
Thus, we get that $\theta_0=-\phi_0\mbox{ or }\theta_0=-\phi_0+\pi \mbox{ or }\theta_0=\phi_0+\pi.$ 

\textbf{Case 1:} If $\theta_0=-\phi_0$ then 
\begin{eqnarray*}
h(\theta_0,\phi_0)&=&3+2(2cos \theta_0 + \alpha cos 2\theta_0)\\
&=&4\alpha cos^2\theta_0 + 4cos \theta_0 +3-2\alpha\\
&=& g(t),\mbox{ say,}
\end{eqnarray*}
where $t=cos \theta_0.$

If $\alpha> -1/2$ then there is no critical point of $g$ in $(0,1).$ 
Hence the maximum is $3-2\alpha$ or $|7+2\alpha|.$ 
Between these, in the case when $\alpha\in (-1/2,0),$ clearly, $3-2\alpha$ is bigger. 
If $\alpha\leq -1/2$ then the maximum of $g(t)$ is among 
$3-2\alpha-1/\alpha$ or $|7+2\alpha|$ or $3-2\alpha.$ 
A straightforward computation shows that 
$3-2\alpha-1/\alpha$ is bigger than the other two values. 
Therefore 
we have the following 

\[M_{\mathscr{B}_\alpha}=
\begin{cases}
3-2\alpha, & \alpha> -\frac{1}{2}\\
3-2\alpha-1/\alpha, & \alpha\leq -\frac{1}{2}.\\
\end{cases}
\] 
Computations done till now leads to the following graph of the ratio $\mathcal{Q}:=$ $\frac{M_{\mathscr{B}_\alpha}}{\|p_{\!_{\mathscr{B}_\alpha}}\|_{\D^3,\infty}}$
\[\mathcal{Q}=
\begin{cases}
\frac{3-2\alpha}{7+2\alpha}, & \alpha \in (-1/2,0)\\
\frac{3-2\alpha-1/\alpha}{7+2\alpha}, & \alpha \in [-1,-1/2]\\
\frac{3-2\alpha-1/\alpha}{3-2\alpha}, & \alpha \in (-\infty,-1).
\end{cases}
\]
The ratio 
$\mathcal{Q}$ 
is increasing in $(-\infty,-1)$ and decreasing in $(-1,0).$ 
Thus maximum of the ratio 
$\mathcal{Q}$ 
is at $-1$ and hence \VK polynomial is the best in the class we have specified.

\textbf{Acknowledgement:} We are very grateful to Prof. Gadadhar Misra and Prof. Gilles Pisier for several fruitful discussions and suggestions. The authors also express their sincere gratitude to Prof. Sameer Chavan and Prof. Parasar Mohanty for their constant support. We also thank the referee and the editor for several constructive suggestions which significantly improved the presentation of the paper.

\bibliographystyle{amsalpha}\bibliography{On_A_Question_Of_Varopoulos_Final_Version}
  
\end{document}